\documentclass{article}  
\usepackage{amsthm}
\usepackage{amsmath,amssymb,bm}

\usepackage[margin=1in]{geometry}
\usepackage{graphicx}
\begin{document}

\newcommand{\la}{\lesssim}
\newcommand{\ga}{\gtrsim}
\newcommand{\bE}{\ensuremath{\mathbf{E}}}
\newtheorem{theorem}{Theorem}[section]
\newtheorem{proposition}[theorem]{Proposition}
\newtheorem{conjecture}[theorem]{Conjecture}
\newtheorem{observation}[theorem]{Observation}
\newtheorem{lemma}[theorem]{Lemma}
\newtheorem{corollary}[theorem]{Corollary}
\newtheorem{definition}[theorem]{Definition}

\title{Some results on chromatic number as a function of triangle count}

\author{David G. Harris\thanks{Department of Computer Science, University of Maryland, 
College Park, MD 20742.
Email: \texttt{davidgharris29@gmail.com}}}

\maketitle
\begin{abstract}
A variety of powerful extremal results have been shown for the chromatic number of triangle-free graphs. Three noteworthy bounds are in terms of the number of vertices, edges, and maximum degree given by Poljak \& Tuza (1994), and Johansson.  There have been comparatively fewer works extending these types of bounds to graphs with a small number of triangles. One noteworthy exception is a result of Alon et. al (1999) bounding the chromatic number for graphs with low degree and few triangles per vertex; this bound is nearly the same as for triangle-free graphs. This type of parametrization is much less rigid, and has appeared in dozens of combinatorial constructions.

In this paper, we show a similar type of result for $\chi(G)$ as a function of the number of vertices $n$,  the number of edges $m$, as well as the triangle count (both local and global measures). Our results smoothly interpolate between the generic bounds true for all graphs and bounds for triangle-free graphs. Our results are tight for most of these cases; we show how an open problem regarding fractional chromatic number and degeneracy in triangle-free graphs can resolve the small remaining gap in our bounds.
\end{abstract}

\section{Introduction}
In this paper, we examine extremal bounds on the chromatic number of an undirected graph $G = (V,E)$ which has relatively few triangles (a triangle is a triple of vertices $x,y,z \in V$ where all three edges $(x,y), (x,z), (y,z)$ are present in $G$). For way of motivation,  a variety of powerful results have been shown for the chromatic number of \emph{triangle-free} graphs. One noteworthy bound, which we explore further in this paper, is in terms of the number of vertices or edges:\footnote{See also \cite{gimbel} and \cite{nilli} for crisper proofs of these results.}
\begin{theorem}[\cite{poljak-tuza}]
\label{triangle-free-thm1}
Suppose $G$ is triangle-free with $n$ vertices and $m$ edges. Then
$$
\chi(G) \leq O \Bigl( \min(\sqrt{ \frac{n}{\log n}}, \frac{m^{1/3}}{\log^{2/3} m}) \Bigr)
$$
\end{theorem}

Another powerful bound for triangle-free graphs can be given in terms of the maximum degree:\footnote{Theorem~\ref{triangle-free-d} is attributed to Johansson, as attributed by \cite{molloy}. See \cite{molloy2} for a more recent proof, which also gives bounds on the constant term.}
\begin{theorem}
\label{triangle-free-d}
Suppose $G$ is triangle-free and has maximum degree $d$. Then $\chi(G) \leq O \Bigl( \frac{d}{\log d} \Bigr)$
\end{theorem}

These bounds are much smaller than would be possible for a generic graph (with no restriction on the number of triangles); in those cases one can only show the bounds
$$
\chi(G) \leq \min(n, \sqrt{m}, d+1)
$$

The requirement that a graph has no triangles is quite rigid. There have been comparatively fewer works which relax the triangle-free condition to allow a small number of triangles. One noteworthy exception to this is the result of \cite{aks}:
\begin{theorem}[\cite{aks}]
\label{aks-thm}
Suppose the graph $G$ has maximum degree $d$ and each vertex incident on at most $y$ triangles, where $1 \leq y \leq d^2/2$. Then 
$$
\chi(G) \leq O( \frac{d}{\log(d^2/y)} )
$$
\end{theorem}

Thus, for instance, if $y < d^{2 - \Omega(1)}$, then the worst-case behavior of $\chi(G)$ is roughly the same as if $G$ had no triangles at all.  The result of \cite{aks} has been used in dozens of combinatorial constructions.  Roughly speaking, if $G$ is produced in a somewhat ``random'' or ``generic'' way, then $G$ will have relatively few triangles; only a few extremal cases (such as a graph containing a $d$-clique) give the full triangle count. 

In this paper, we show upper bounds for $\chi(G)$ as a function of the number of vertices $n$,  the number of edges $m$, as well as the triangle count. Our results smoothly interpolate between the generic bounds true for all graphs, and the triangle-free bounds of Theorems~\ref{triangle-free-thm1}. We will see here as well that graphs with few triangles have nearly the same behavior as graphs with no triangles. This is convenient because triangle counts (local and global) can be computed easily in polynomial time; this is quite different from computing $\chi(G)$ directly (which is NP-hard).

\begin{definition}
We say that a vertex $v \in V$ has \emph{local triangle count} $y$, if there are $y$ pairs of vertices $u, w \in V$ such that $(v,u), (v,w), (u,w) \in E$. We say that $G$ has \emph{local triangle bound} $y$ if every vertex has local triangle count at most $y$.
\end{definition}

We shall show the following bounds on the chromatic number:
\begin{theorem}
Suppose a graph $G$ has $m$ edges, $n$ vertices, $t$ triangles, and local triangle bound $y$. Then
$$
\chi(G) \leq \min(a_1, a_2, a_3, a_4, a_5, a_6)
$$
where
\begin{align*}
&a_1 = O \bigl( \sqrt{ \frac{n}{\log n}}  + \frac{n^{1/3} y^{1/3}}{\log^{2/3} (n^2/y)} \bigr) && a_2 = O \bigl( \frac{m^{1/3}}{\log^{2/3} m} + \frac{m^{1/4} y^{1/4}}{\log^{3/4} (m/y)} \bigr) \\
&a_3 = O \bigl( \sqrt{\frac{n}{\log n}} + \frac{t^{1/3} \log \log(t^2/y^3)}{\log^{2/3}(t^2/y^3)} \bigr) && a_4 = O \bigl( \frac{m^{1/3}}{\log^{2/3} m} + \frac{t^{1/3} \log \log(t^2/y^3)}{\log^{2/3}(t^2/y^3)} \bigr) \\
&a_5 = O(\sqrt{\frac{n}{\log n}}) + (6^{1/3} + o(1)) t^{1/3} && a_6= O(\frac{m^{1/3}}{\log^{2/3} m}) + (6^{1/3} + o(1)) t^{1/3} 
\end{align*}
\end{theorem}

We show in Section~\ref{lb-sec} that the bounds $a_1, a_2$ are tight up to constant factors, and the bounds $a_5, a_6$ are tight up to second-order terms. The bounds $a_3, a_4$ are tight up to factors of $\log \log (t^2/y^3)$; we will show in Section~\ref{conj-sec} that this small gap can be resolved via a conjecture on fractional chromatic number and degeneracy in triangle-free graphs.

We also remark that all our bounds on chromatic number can be realized constructively; that is, there are randomized algorithms which construct such colorings in expected polynomial time. We do not discuss any further algorithmic issues in this paper.

\subsection{Notation and preliminaries}
We write $f \la g$ to mean that $f = O(g)$. We write $f \approx g$ if $f = \Theta(g)$.  Given a graph $G$ on a vertex set $V$ and a subset $B \subseteq V$, we use $\chi(B)$ as a shorthand for $\chi(G[B])$, where $G[B]$ denotes the subgraph induced on $B$.  We let $N(v)$ denote the neighborhood of $v$, i.e. the set of vertices $w$ with $(w,v) \in E$. The degree of $v$ is $\text{deg}(v) = |N(v)|$.

We use the expression $\log$ to denote the \emph{truncated logarithm}, i.e.
$$
\log(x) = \begin{cases}
\ln(x) & x > e \\
1 & x \leq e
\end{cases}
$$

Note that, due to this definition, iterated logarithms are always well-defined. In addition, we have the following useful bound:
\begin{observation}
  \label{obs1}
  Suppose $x_0 \leq x \leq x_1$. Then for $a \leq 1$ and $y \geq 0$ we have
  $$
  x_0 \log^a(y/x_0) \leq x \log^a(y/x) \leq x_1 \log^a(y/x_1)
  $$
\end{observation}
\begin{proof}
  We claim that $x \log^a(y/x)$ is an increasing function of $x$. When $x \geq y/e$, we have $x \log^a(y/x) = x$ and so this is clearly true. Otherwise, we have $x \log^a(y/x) = x (\ln(y/x))^a$, which has derivative $(\ln(y/x))^{a-1} (\ln(y/x) - a)$; this is positive for $y > e x$ and $a \leq 1$.
  \end{proof}
  
Our constructions will make use of a few previous results for coloring and independent sets, which we summarize here. We have already mentioned the theorem of \cite{aks} for chromatic number of locally-sparse graphs as a function of maximum degree. We restate it here, taking advantage of our truncated logarithm and asymptote notations:
\begin{theorem}[\cite{aks}]
\label{aks-thm2}
Suppose $G$ has maximum degree $d$ and local triangle bound $y$. Then $\chi(G) \la \frac{d}{\log(d^2/y)}$.
\end{theorem}

Vu \cite{vu} extended Theorem~\ref{aks-thm} to list-chromatic number:
\begin{theorem}[\cite{vu}]
\label{aks-thm3}
Suppose $G$ has maximum degree $d$ and local triangle bound $y$. If each vertex has a palette of size $\frac{c d}{\log(d^2/y)}$, where $c$ is a universal constant, then $G$ can be list-colored.
\end{theorem}

We also use a reformulated version of Tur\'{a}n's Theorem (the original version is the equivalent statement for the complement graph.)
\begin{theorem}[Reformulated Tur\'{a}n's Theorem]
\label{turan-thm}
Suppose $G$ is a graph with $n$ vertices and $m$ edges. Then $G$ has an independent set of size at least $\frac{n}{2m/n + 1}$.
\end{theorem}

\section{Bounds in terms of $n$}
We begin by showing bounds on $\chi(G)$ as a function of triangle information (local and global) as well the vertex count $n$. These generalize the result of \cite{ajtai} for triangle-free graphs. Notably, when $G$ has a relatively small number of triangles, we show that $\chi(G)$ has approximately the same worst-case behavior as if $G$ had no triangles at all; namely, we have the bound $\chi(G) \la \sqrt{\frac{n}{\log n}}$.

\begin{theorem}
\label{prop0}
Suppose $G$ has $n$ vertices and local triangle bound $y$. Then
$$
\chi(G) \la \sqrt{ \frac{n}{\log n}}  + \frac{n^{1/3} y^{1/3}}{\log^{2/3} (n^2/y)}
$$
\end{theorem}
\begin{proof}
  Let $f = \log(n^2/y)$
  Our plan is to repeatedly remove independent sets from $G$, as long as its maximum degree exceeds a a parameter $d$.

  Initially let $G_0 = G$. Now, repeat the following for $i = 0, 1, 2, \dots, k$: if all the vertices of $G_i$ have degree at most $d$, then abort the process and color the residual graph $G_i$ using Theorem~\ref{aks-thm2}; this requires $O( \frac{d}{\log(d^2/y)})$ colors.  Otherwise, select some vertex $v_i$ which has degree $> d$ in $G_i$. Thus $G[N(v_i)]$ has at least $d+1$ vertices and has at most $y$ edges. By Theorem~\ref{turan-thm}, there is an independent set $I_i \subseteq N(v_i)$ of size  at least $\frac{d}{2y/d + 1}$. We color all the vertices of $I_i$ with a new color, and let $G_{i+1} = G_i - I_i$. Overall, this process uses $O( \frac{d}{\log (d^2/y)} + k)$ colors.

Since the sets $I_1, \dots, I_k$ are all disjoint, we have $k \leq \frac{n (2y/d + 1)}{d} = \frac{2 n y}{d^2} + \frac{n}{d}$ and so
$$
\chi(G) \la \frac{d}{\log(d^2/y)} + \frac{n y}{d^2} + \frac{n}{d}
$$

When $y \leq \sqrt{n \log n}$, we set $d = \sqrt{n \log n}$ and get
$$
\frac{d}{\log (d^2/y)} + \frac{n y}{d^2} + \frac{n}{d} \leq \frac{\sqrt{n \log n}}{\log(\sqrt{n \log n})} + \frac{ 2 n}{\sqrt{n \log n}} + \frac{n}{\sqrt{n \log n}} \la \sqrt{n/\log n}
$$

When $y > \sqrt{n \log n}$, we set $d = (n y f)^{1/3}$. We estimate the $\log(d^2/y)$ term as:
$$
\log(d^2/y) = \log(n^{2/3} f^{2/3} y^{-1/3}) = \log( (n^2/y)^{1/3} \log^{2/3} (n^2/y)) \approx \log(n^2/y) = f
$$
and so
$$
\frac{d}{\log (d^2/y)} + \frac{n y}{d^2} + \frac{n}{d} \la \frac{d}{f} + \frac{n^{1/3} y^{1/3}}{f^{2/3}} + \frac{n^{2/3}}{(y \log(n^2/y))^{1/3}}
$$

As $y > \sqrt{n \log n}$, Observation~\ref{obs1} gives
$$
\frac{n^{2/3}}{(y \log(n^2/y))^{1/3}} \leq \frac{n^{2/3}}{(\sqrt{n \log n} \log( \frac{n^2}{\sqrt{n \log n}}))^{1/3}} = \sqrt{\frac{n}{\log n}},
$$
which completes the proof.
\end{proof}

Next, we will show bounds on $\chi(G)$ given information both about the global triangle count and the local triangle bound. We begin with a useful lemma.
\begin{lemma}
\label{ff-lemma}
Suppose that $G$ has local triangle bound $y$, and there is a partition of the vertices $V = A_1 \sqcup A_2 \sqcup \dots \sqcup A_k$ such that, for all $1 \leq i \leq j \leq k$, every $v \in A_i$ satisfies $|N(v) \cap A_j| \leq d x^{i - j}$ for some parameters $d, x \geq 1$. Then 
$$
\chi(G) \la \frac{d (1 +\frac{\log \log(d^2/y)}{\log x})}{\log(d^2/y)}
$$

In particular, if $x$ is a constant, then
$$
\chi(G) \la \frac{d \log \log(d^2/y)}{\log(d^2/y)}
$$
\end{lemma}
\begin{proof}
  Define $f = \log(d^2/y) \geq 1$. We will first show that this result holds for $x \geq 2 f$. In this case, we have $\log \log(d^2/y)/\log x \leq 1$ and so we need to show that $\chi(G) \la d/f$.

  We will allocate a palette of size $c d/f$ to each vertex $v$ for some constant $c$ to be determined. We proceed for $j = k, k-1, \dots, 1$, attempting to list-color $G[A_j]$; at stage $j$, we remove from the palette of each vertex $v \in A_j$ all the colors already used by its neighbors in $A_{j+1}, \dots, A_k$. As $f \geq 1$, the total number of such neighbors is at most
$$
  d x^{-1} + d x^{-2} + d x^{-3} + \dots + d x^{-k+j} \leq d (\frac{1}{2f} + \frac{1}{(2f)^2} + \frac{1}{(2f)^3} + \dots) \leq d/f
$$

Thus, every vertex $v \in A_j$ has a residual palette size of at least $c d / f - d/f = (c - 1) d/f$. If $c$ is a sufficiently large constant, then Theorem~\ref{aks-thm3} allows us to color $A_j$ with these palettes. So each set $A_k, A_{k-1}, \dots, A_1$ gets successfully colored in turn.

We next consider the case that $x < 2 f$. Let us set $s = \lceil \frac{\log(2 f)}{\log x} \rceil$ and for each $i = 1, \dots, s$ define $B_i = A_i \cup A_{i+s} \cup A_{i+2 s} \cup \dots$. Observe that each $G[B_i]$ satisfies the hypotheses of this lemma with the vertex-partition $A_i, A_{i+s}, A_{i+2 s}, \dots$ and with parameters $d' = d, x' = x^s$. Since $x' \geq 2 f$, we have already proved that the lemma holds in the situation and so
$$
\chi(B_i) \la \frac{d' (1 + \frac{\log \log((d')^2/y)}{\log(x')})}{\log( (d')^2/y)} \approx \frac{d}{f}
$$

So
$$
\chi(G) \leq \chi(B_1) + \dots + \chi(B_s) \la s d/f \leq \frac{(1 + \frac{\log(2 f)}{\log x}) d}{f} \la \frac{d (1 + \frac{\log f}{\log x})}{f}
$$
\end{proof}

\begin{theorem}
\label{ttprop2}
Suppose $G$ has $n$ vertices, $t$ triangles, and local triangle bound $y$. Then
$$
\chi(G) \la  \sqrt{ \frac{n}{\log n}} + \frac{t^{1/3} \log \log (t^2/y^3)}{\log^{2/3}(t^2/y^3)}
$$
\end{theorem}
\begin{proof}
Let $A_i$ denote the set of vertices in $G$ with between $2^i$ and $2^{i+1}$ triangles and let $d$ be a parameter which we will specify shortly. We let $A_{-1}$ denote the vertices incident on zero triangles. We also set $r = \lceil \log_2 d \rceil, f = \log(t^2/y^3)$, and $\ell = \lceil \log_2 y \rceil$.

We will color $G$ in three stages. In the first stage, whenever there is some vertex $v \in A_i$ such that $v$ has more than $2^{(i-j)/2} d$ neighbors in $A_j$ where $r \leq i \leq j$, we apply Theorem~\ref{turan-thm} to obtain an independent set $I \subseteq N(v) \cap A_j$.  We assign each vertex in $I$ to one new color and remove $I$ from the graph. Note that this process of removing vertices may cause the membership of the sets $A_{k}$ to change.

The graph $G[N(v) \cap A_j]$ has at least $2^{(i-j)/2} d$ vertices. By definition of $A_j$, it has at most $2^{j+1}$ edges.  Thus, the set $I$ has size
$$
|I| \geq \frac{2^{(i-j)/2} d}{1 + \frac{2^{i+2}}{2^{(i-j)/2} d}}
$$

Observe that $2^{i+2} \geq 2^{r+2} \geq 4 d$. Since $i \leq j$, we therefore have $\frac{2^{i+2}}{2^{(i-j)/2} d} \geq 1$ and so
$$
|I| \ga \frac{2^{(i-j)/2} d}{\frac{2^{i+2}}{2^{(i-j)/2} d}} = 2^{-j} d^2 / 4
$$

By definition of $A_j$, each $w \in I$ is incident on at least $2^j$ triangles. Thus, when we remove $I$ from the graph, we remove at least $|I| 2^j  \ga d^2$ triangles from $G$. Consequently, the total number of colors used in the first phase is at most $O(t/d^2)$.

Next, suppose that the first phase has finished, and there no more vertices $v \in A_i$ which have more than $2^{(i-j)/2} d$ neighbors in $A_j$ for any $r \leq i \leq j$. The graph $G[A_{r} \cup \dots \cup A_{\ell}]$, with the corresponding partition $A_r, \dots, A_{\ell}$, satisfies the requirement of Lemma~\ref{ff-lemma} with parameter $x = \sqrt{2}$, so it can be colored using $O( \frac{d \log \log (d^2/y)}{\log (d^2/y)})$ colors.

Finally, we apply Theorem~\ref{prop0} to color $G[A_{-1} \cup A_0 \cup \dots \cup A_{r-1}]$. This graph has at most $2^{r} \la d$ triangles per vertex, hence this requires $O( \sqrt{\frac{n}{\log n}} + \frac{(n d)^{1/3}}{\log^{2/3}(n^2/d)})$ colors.

Putting all three terms together,
\begin{equation}
  \label{hg5}
\chi(G) \la  \frac{t}{d^2} +  \frac{d \log \log (d^2/y)}{\log (d^2/y)} + \Bigl( \sqrt{\frac{n}{\log n}} +  \frac{(n d)^{1/3}}{\log^{2/3}(n^2/d)} \Bigr)
\end{equation}

Now set $d = (f t)^{1/3} + \sqrt{n / \log n}$. We  observe the following useful bound:
$$
\log(d^2/y) \geq \log \Bigl( \frac{(f t)^{2/3}}{y} \Bigr)  = \log \Bigl( \log^{2/3} (t^2/y^3) ( t^{2/3}/y) \Bigr) \approx \log( t^2/y^3) = f
$$

Clearly $f \la \log n$. Also, since $t \leq n^3$, we have $d \la (n^3 \log n)^{1/3} + \sqrt{n / \log n}$ and  thus $\log(n^2/d) \ga \log n$. Using these bounds, we bound the summands of (\ref{hg5}) in turn:
\begin{align*}
  \frac{t}{d^2} &\leq \frac{t}{ (f t)^{2/3}} = \frac{t^{1/3}}{f^{2/3}} \\
\frac{d \log \log (d^2/y)}{\log(d^2/y)} &\la \frac{d \log f}{f} = \frac{t^{1/3} \log f}{f^{2/3}} + \frac{\sqrt{n} \log f}{\sqrt{\log n} f} \leq \frac{t^{1/3} \log f}{f^{2/3}} + \sqrt{n/\log n} \\
\frac{(n d)^{1/3}}{\log^{2/3}(n^2/d)} &\la \frac{\sqrt{n}}{\log^{5/6} n} + \frac{n^{1/3} t^{1/9} f^{1/9}}{\log^{2/3} n}
\end{align*}

All but one of these terms are clearly bounded by $O( \sqrt{ \frac{n}{\log n}} + \frac{t^{1/3} \log f}{f})$ as we have claimed. The one exception is the final term $\frac{n^{1/3} t^{1/9} f^{1/9}}{\log^{2/3} n}$. For this term,  we consider two cases. First, when $t \leq n^{3/2} \sqrt{\log n}$, we have
$$
\frac{n^{1/3} t^{1/9} f^{1/9}}{\log^{2/3} n} \leq \frac{n^{1/3} ( n^{3/2} \log^{1/2} n)^{1/9} (\log n)^{1/9}}{\log^{2/3} n} = \sqrt{n/\log n}
$$

When $t > n^{3/2} \sqrt{\log n}$, then \begin{align*}
  \frac{n^{1/3} t^{1/9} f^{1/9}}{\log^{2/3} n} &\leq \frac{t^{1/3}}{f^{2/3}} \times \frac{f^{7/9} n^{1/3}}{t^{2/9} \log^{2/3} n} \leq \frac{t^{1/3}}{f^{2/3}} \times \frac{(\log n)^{7/9} n^{1/3}}{t^{2/9} \log^{2/3} n} = \frac{t^{1/3}}{f^{2/3}} \times \frac{(\log n)^{1/9} n^{1/3}}{t^{2/9}} \\
  &\leq \frac{t^{1/3}}{f^{2/3}} \times \frac{(\log n)^{1/9} n^{1/3}}{(n^{3/2} \log^{1/2} n)^{2/9}} = \frac{t^{1/3}}{f^{2/3}}
\end{align*}
Thus, in either case, we have shown that
$$
\frac{n^{1/3} t^{1/9} f^{1/9}}{\log^{2/3} n} \la \sqrt{n/\log n} +  \frac{t^{1/3}}{f^{2/3}}
$$
completing the proof.
\end{proof}

\subsection{A result on independence number}
Recently, Bohman \& Mubayi \cite{bohman} have investigated the relation between independence number $\alpha(G)$ and the number of copies of $K_s$ in a graph $G$. In particular, for $s = 3$, they discuss extremal bounds relating $\alpha(G)$ and triangle count. As we show next, these results may be obtained as immediate corollaries of Theorem~\ref{prop0}. (Note that Theorem~\ref{prop0} requires the use of heavy-duty results of Johansson and \cite{aks}, whereas \cite{bohman} uses more elementary methods.)

\begin{corollary}[\cite{bohman}]
If $G$ has $n$ vertices and $t$ triangles, then
$$
\alpha(G) \ga \begin{cases}
\sqrt{n \log n} & \text{if $t \leq n^{3/2} \sqrt{\log n}$} \\
(n/t^{1/3}) \log^{2/3} (n/t^{1/3}) & \text{if $t \geq n^{3/2} \sqrt{\log n}$}
\end{cases}
$$
\end{corollary}
\begin{proof}
  Let $S$ denote the set of vertices which are incident upon at most $y = 10 t/n$ triangles.
  Applying Theorem~\ref{prop0} to $G[S]$, we have
  \begin{equation}
    \label{tgac1}
\chi(G[S]) \la \sqrt{ \frac{n}{\log n} } + \frac{t^{1/3}}{\log^{2/3} (n^3/t)}
\end{equation}

Since the average triangle count is $3 t/n$, we have $|S| \approx n$, and so $G[S]$ has an independent set of size at least $n/\chi(G[S])$; simple calculations show that the bound (\ref{tgac1}) achieves the claimed result.
\end{proof}

\section{Bounds in terms of $m$}
In this section, we show some bounds in terms of the edge count $m$, as well as triangle count (local and global). These generalize results of \cite{nilli} and \cite{gimbel}, which show similar bounds for the chromatic number as a function of $m$ for triangle-free graphs. As in the vertex-based bounds, we will show that when $t, y$ are small, the worst-case behavior for $\chi(G)$ is essentially the same as if $G$ had no triangle at all, namely $\chi(G) \la \frac{m^{1/3}}{\log^{2/3} m}$.

\begin{theorem}
\label{prop0a}
Suppose $G$ has $m$ edges and local triangle bound $y$. Then
$$
\chi(G) \la \frac{m^{1/3}}{\log^{2/3} m} + \frac{m^{1/4} y^{1/4}}{\log^{3/4} (m/y)}
$$
\end{theorem}
\begin{proof}
  If $y \geq m$, then this simply asserts that $\chi(G) \la \sqrt{m}$, which holds for any graph. So we assume $y < m$. Now let $f = \log (m/y)$, let $A$ denote the vertices of degree greater than $d$, and let  $B$ denote the vertices of degree $\leq d$, where $d$ is some parameter to be chosen. Note that $|A| \leq 2 m/d$. Applying Theorem~\ref{prop0} to $G[A]$ and Theorem~\ref{aks-thm2} to $G[B]$, we have
\begin{equation}
  \label{ghh1}
  \chi(G) \la \chi(A) + \chi(B) \la \sqrt{ \frac{m/d}{\log(m/d)}} + \frac{(m/d)^{1/3} y^{1/3}}{\log^{2/3} ( (m/d)^2/y ) } + \frac{d}{\log (d^2/y)}
\end{equation}

If $y < (m \log m)^{1/3}$, then set $d = (m \log m)^{1/3}$. We simplify the log terms in (\ref{ghh1}):
\begin{align*}
\log(d^2/y) &\geq \log( \frac{(m \log m)^{2/3}}{(m \log m)^{1/3}} ) \approx \log m \\
\log(m/d) &= \log(\frac{m}{(m \log m)^{1/3}}) \approx \log m \\
\log((m/d)^2/y) &\geq \log( \frac{m^2}{(m \log m)^{2/3} \times (m \log m)^{1/3}} ) \approx \log m
\end{align*}
Substituting the  bounds on the logarithm terms into (\ref{ghh1}) we get:
\begin{align*}
  \chi(G) &\la \sqrt{ \frac{ m / (m \log m)^{1/3} }{\log m} } + \frac{ (\frac{m}{(m \log m)^{1/3}})^{1/3} ( (m \log m)^{1/3} )^{1/3} }{\log^{2/3} m} + \frac{ (m \log m)^{1/3}}{\log m} = \frac{3 m^{1/3}}{\log^{2/3} m}
\end{align*}

If $y \geq (m \log m)^{1/3}$, then set $d = (m y f)^{1/4}$. As $y \leq m$, Observation~\ref{obs1} gives $y f = y \log(m/y) \leq m$ and so $d \leq \sqrt{m}$.  We thus simplify the log terms as:
\begin{align*} 
  \log(d^2/y) &= \log( (m y f)^{1/2} / y ) = \log( (m/y)^{1/2} \log^{1/2}(m/y) ) \approx \log(m/y) = f \\
\log(m/d) &\geq \log(\frac{m}{\sqrt{m}}) \approx \log m \\
\log((m/d)^2/y) & \geq \log( \frac{m^2}{(\sqrt{m})^2 y} ) = \log(\frac{m}{y}) \approx f
\end{align*}

Substituting the  bounds on the logarithm terms into (\ref{ghh1}) we get:
\begin{align*}
  \chi(G) &\la \sqrt{ \frac{ m / (m \log m)^{1/3} }{\log m} } + \frac{(m/( m y f)^{1/4})^{1/3} y^{1/3}}{f^{2/3}} + \frac{(m y f)^{1/4}}{f} = \frac{m^{1/3}}{\log^{2/3} m} + \frac{2 (m y)^{1/4}}{f^{3/4}} 
\end{align*}
\end{proof}

\begin{theorem}
\label{ttprop3}
Suppose $G$ has $m$ edges, $t$ triangles, and local triangle bound $y$. Then
$$
\chi(G) \la \frac{m^{1/3}}{\log^{2/3} m} + \frac{t^{1/3} \log \log(t^2/y^3)}{\log^{2/3}(t^2/y^3)}
$$
\end{theorem}
\begin{proof}
We assume $t > 0$, as otherwise this is simply Theorem~\ref{triangle-free-thm1}.  Let $f = \log(t^2/y^3)$. Let $A$ denote the set of vertices incident on at least $z = \frac{t^{2/3} \log^{1/3} m}{m^{1/3}}$ triangles, and let $B$ denote the remaining vertices. Note that $|A| \leq 3 t/z$ and $G[B]$ has local triangle bound $z$. Applying Theorem~\ref{ttprop2} to $G[A]$ and Theorem~\ref{prop0a} to $G[B]$ we get
  \begin{equation}
    \label{hgg6}
  \chi(G) \leq \chi(A) + \chi(B) \la \frac{t^{1/3} \log f}{f^{2/3}} + \sqrt{ \frac{t/z}{\log(t/z)} } + \frac{m^{1/3}}{\log^{2/3} m} + \frac{(m z)^{1/4}}{\log^{3/4}(m/z)}
  \end{equation}

  Noting that $1 \leq t \leq m^{3/2}$, we compute these log terms as:
  \begin{align*}
    \log(t/z) &= \log( \frac{(m t)^{1/3}}{\log^{1/3} m} ) \geq \log( \frac{m^{1/3}}{\log^{1/3} m} ) \approx \log m \\
    \log(m/z) &= \log( \frac{m^{4/3}}{(\log m)^{1/3} t^{2/3}} ) \geq \log( \frac{m^{4/3}}{(\log m)^{1/3} (m^{3/2})^{2/3}} ) = \log( \frac{m^{1/3}}{\log^{1/3} m}) \approx \log m
  \end{align*}

  Substituting these bounds into (\ref{hgg6}) gives
$$
    \chi(G) \la \frac{t^{1/3} \log f}{f} + \sqrt{ \frac{t/z}{\log m} } + \frac{m^{1/3}}{\log^{2/3} m} + \frac{(m z)^{1/4}}{\log^{3/4} m} = \frac{t^{1/3} \log f}{f} + \frac{m^{1/3}}{\log^{2/3} m} + \frac{ 2 m^{1/6} t^{1/6}}{\log^{2/3} m}
    $$

    Now observe that (using the inequality $a b \leq a^2 + b^2$), we have
    $$
    \frac{m^{1/6} t^{1/6}}{\log^{2/3} m} \leq \frac{m^{1/3}}{\log^{2/3} m} + \frac{t^{1/3}}{\log^{2/3} m} \leq  \frac{m^{1/3}}{\log^{2/3} m} + \frac{t^{1/3}}{f^{2/3}}
    $$
    which completes the proof.
\end{proof}

\section{Lower bounds}
\label{lb-sec}
We next show matching lower bounds. We will show that Theorems~\ref{prop0} and \ref{prop0a} are tight for all admissible values of $m,n,y$ up to constant factors, while Theorem~\ref{ttprop2} is tight up to factors of $\log \log(t^2/y^3)$ for all admissible values of $m,n,y,t$. The situation for Theorem~\ref{ttprop3} is slightly more complicated; in general, the bounds given by Theorem~\ref{ttprop3} and Theorem~\ref{prop0a} are incomparable. For a given value of $m,n,y,t$, we show that either Theorem~\ref{ttprop3} or Theorem~\ref{prop0a} is tight (the latter up to a factor of $\log \log(t^2/y^3)$). The lower bounds apply even for the fractional chromatic number (see Section~\ref{conj-sec} for the definition and more details).

We begin by recalling a result of \cite{kim}:
\begin{theorem}[\cite{kim}]
\label{kim-thm}
For any integer $n \geq 1$, there exists a graph $H_n$ on $n$ vertices with the following properties:
\begin{enumerate}
\item[(A1)] $H_n$ is triangle-free
\item[(A2)] Each vertex has degree at most $O(\sqrt{n \log n})$
\item[(A3)] $\alpha(H_n) \leq  O(\sqrt{n \log n})$  (where $\alpha(G)$ denotes the size of the maximum independent set of $G$)
\item[(A4)] $H_n$ has chromatic number $\chi(H_n) \geq \Omega(\sqrt{\frac{n}{\log n}})$.
\end{enumerate}
\end{theorem}

Following a strategy of \cite{aks}, we construct a blow-up of $H_n$ with the complete clique $K_i$, to give:
\begin{proposition}
\label{kim-thm2}
For any integers $k,i  \geq 1$, there is a graph $H_{k,i}$ with the following properties:
\begin{enumerate}
\item[(B1)] $H_{k,i}$ contains $O(k i)$ vertices
\item[(B2)] $H_{k,i}$ has local triangle bound $O(i^2 \sqrt{k \log k})$.
\item[(B3)] Each vertex has degree at most $O(i \sqrt{k \log k})$
\item[(B4)] $\alpha(H_{k,i}) \leq O(\sqrt{k \log k})$.
\item[(B5)] The fractional chromatic number of $H_{k,i}$ satisfies $\chi_{f} (H_{k,i}) \ga i \sqrt{ \frac{k}{\log k}}$.
\end{enumerate}
\end{proposition}
\begin{proof}
  We replace each vertex of $H_k$ with an $i$-clique. For every edge $(x,y) \in H_{n}$, we place an edge between all the corresponding copies of $x, y$ in $H_{k,i}$, a total of $i^2$ edges. Now (B1) follows immediately from the fact that $H_k$ contains $k$ vertices. To show (B3), consider a vertex $v \in H_k$ and a corresponding vertex $v'$ in $H_{k,i}$. The vertex $v'$ has $i-1$ edges going to the other vertices in the clique corresponding to $v$. For each neighbor $w$ of $v$ in $H_k$, the vertex $v'$ has $i$ edges (one for each vertex $w'$ in the clique corresponding to $w$). Overall, it has $(i-1) + \deg(v) i$ neighbors. By(A2), $\deg(v) \leq O(\sqrt{k \log k})$ and so this is $O( (i-1) + i \sqrt{k \log k})$. Since $k \geq 1$, the first term $i-1$ is negligible compared to the second one.
  
To show (B2), consider a vertex $x' \in H_{k,i}$ corresponding to $x \in H_k$. We want to count the triangles $x', y', z'$ in $H_{k,i}$, where $y', z'$ correspond to $y, z \in H_k$. We cannot have $x,y,z$ be distinct as otherwise $(x,y), (y,z), (x,z)$ would be a triangle in $H_k$. When $y = z = x$, the total number of such triangles is at most $i^2$ (since $y', z'$ must lie in the same clique as $x'$). When $y = z \neq x$, then there must be an edge in $H_k$ from $x$ to $y$. There are at most $O(\sqrt{k \log k})$ choices of $y$ and once $y$ is fixed, at most $i^2$ choices for $y', z'$.  Finally, when $y = x \neq z$, there must be an edge in $H_k$ from $x$ to $z$. There are at most $O(\sqrt{k \log k})$ choices for $z$ and at most $i^2$ choices for $y', z'$. In total, there are $O(i^2 \sqrt{k \log k})$ triangles involving $x'$.

To show (B4), observe that if $I$ is an independent set of $H_{k,i}$, then all of its vertices must correspond to distinct vertices of $H_k$, and it must correspond to an independent set of $H_k$. So $|I| \leq O(\sqrt{k \log k})$. 

The bound (B5) follows from (B1), (B4) and the bound $\chi_{f} (G) \geq \frac{ |V(G)|}{\alpha(G)}$.
\end{proof}

\textbf{Note on rescaling for Proposition~\ref{kim-thm2}.} The bounds of Proposition~\ref{kim-thm2} will also hold for any real numbers $k,i$ which are bounded uniformly away from $0$, i.e. satisfying $k, i \geq c$ for some constant $c > 0$. To see this, we simply replace the real numbers $k,i$ with the integers $\lceil k \rceil, \lceil i \rceil$. Since $k, i$ are bounded from $0$, the ratios $\lceil k \rceil/k$ and $\lceil i \rceil / i$ are bounded from above by constants. With a slight abuse of our asymptotic notation, we refer to this condition as $i, k \ga 1$. 

\begin{proposition}
\label{lb0}
For any integers $n, y \leq n^2, t \leq n y$, there is a graph $G$ with at most $n$ vertices, at most $t$ triangles, local triangle bound $y$, and such that
$$
\chi(G) \geq \chi_{f}(G) \ga \sqrt{ \frac{n}{\log n}} +  \frac{t^{1/3}}{\log^{2/3} (t^2/y^3)}
$$
\end{proposition}
\begin{proof}
  In order to achieve the bound $\chi_{f}(G) \ga \sqrt{ \frac{n}{\log n}}$, we simply take $G = H_n$. Thus, it suffices to show that we can find such a graph $G$ with
  $$
  \chi_{f}(G) \ga \frac{t^{1/3}}{\log^{2/3} (t^2/y^3)}
  $$
  By rescaling, it suffices to show that a graph has this value of $\chi_{f} (G)$ and $O(n)$ vertices, $O(t)$ triangles, and $O(y)$ local triangle bound. Also, as we have discussed above, we can apply Proposition~\ref{kim-thm2} with real numbers $k, i \ga 1$. Let us define $f = \log(t^2/y^3)$.
  
  We will consider a number of cases.
  
\textbf{Case I: $\bm{t \leq y^{3/2}}$.} In this case, $f = 1$, and we need to show $\chi_{f}(G) \ga t^{1/3}$. Take $G$ to be the complete graph on $t^{1/3}$ vertices. This clearly has $O(t)$ triangles and $O(n)$ vertices. Also, it has a local triangle bound of $t^{2/3} \leq y$.

\textbf{Case II: $\bm{t > y^{3/2}}$ and $\bm{y \leq \sqrt{n \log n}}$}. We have $t \leq n y$ and so by Observation~\ref{obs1} we have
$$
\frac{t^{1/3}}{\log^{2/3}(t^2/y^3)} \leq \frac{(n y)^{1/3}}{\log^{2/3}((n y)^2/y^3)} = \frac{(n y)^{1/3}}{\log^{2/3}(n^2/y)}
$$

Since $y \leq \sqrt{n \log n}$, the logarithm term is bounded by $\log( n^2/y ) \approx \log n$, and so
$$
\frac{t^{1/3}}{f^{2/3}} \la \frac{(n y)^{1/3}}{\log^{2/3} n} \leq \sqrt{ \frac{n}{\log n}}
$$
So in order to show the desired bound on $\chi(G)$, we only need to show that $\chi(G) \ga \sqrt{ n/\log n}$. This is achieved by simply taking $G = H_n$.

\textbf{Case III: $\bm{t > y^{3/2}}$ and $\bm{y > \sqrt{n \log n}}$.}
Then apply Proposition~\ref{kim-thm2} with
$$
i = \frac{y}{f^{1/3} t^{1/3}}, \qquad k = \frac{f^{1/3} t^{4/3}}{y^2}
$$
and take our graph $G$ to be $G = H_{k,i}$.

We must first show that $i,k \ga 1$. For the former term, we use the bound $t \leq n y$ to get:
$$
i = \frac{y}{\log(t^2/y^3)^{1/3} t^{1/3}} \ga \frac{y}{(\log n)^{1/3} (n y)^{1/3}} = \frac{y^{2/3}}{(n \log n)^{1/3}} \geq \frac{(\sqrt{n \log n})^{2/3}}{(n \log n)^{1/3}} = 1
$$

For the latter term, we use the bound $t > y^{3/2}$ to get:
$$
k = \frac{f^{1/3} t^{4/3}}{y^2} \geq \frac{f^{1/3} (y^{3/2})^{4/3}}{y^2} = f^{1/3} \geq 1
$$

We can estimate $\log k$ as:
$$
\log k = \log (f^{1/3} t^{4/3} / y^2) = \log\bigl( (t^2/y^3)^{2/3} \log^{1/3}(t^2/y^3)) \approx \log(t^2/y^3) = f
$$

We next show that $G$ has the desired chromatic number. By (B5), we have
$$
\chi_{f}(H_{k,i}) \ga \frac{i \sqrt{k}}{\sqrt{\log k}} = \frac{t^{1/3}}{f^{1/6} \sqrt{\log k}} \ga \frac{t^{1/3}}{f^{1/6} f^{1/2}} = \frac{t^{1/3}}{f^{2/3}}
$$

Finally we verify that $G$ satisfies the required bounds on its vertex and triangle counts. First, by (B1), the vertex count of $G$ is $O(k i)$, which we bound as
$$
k i = t/y \leq \frac{n y}{y} = n
$$

By (B2), the local triangle bound is $O(i^2 \sqrt{k \log k})$. As $\log k \approx f$, we have $i^2 \sqrt{k \log k} = \frac{y \sqrt{\log k}}{\sqrt{f}} \approx y$.

Finally,  since $H_{k,i}$ has $O(k i)$ vertices and $O(y)$ local triangle bound, it has at most $O(k i y) = O(t)$ triangles.
\end{proof}

\begin{proposition}
Given any integers $m, y \geq 1$ and $t \leq m^{3/2}$, there is a graph $G$ with at most $m$ edges, at most $t$ triangles, local triangle bound $y$, and
$$
\chi(G) \geq \chi_{f}(G) \ga \min \Bigl(  \frac{(m y)^{1/4}}{\log^{3/4}(m/y)}, \frac{t^{1/3}}{\log^{2/3}(t^2/y^3)} \Bigr) + \frac{m^{1/3}}{\log^{2/3} m} 
$$
\end{proposition}
\begin{proof}
Let $f = \log(m/y)$ and let $g = \log(t^2/y^3)$. We begin with a number of preliminary observations. First,   In order to achieve the bound $\chi_{f}(G) \ga \frac{m^{1/3}}{\log^{2/3} m}$, we can simply take $G = H_n$ for $n = m^{2/3}/\log^{1/3} m$. Thus, it suffices to find such a graph $G$ with
\begin{equation}
  \label{sas1}
  \chi_{f}(G) \ga \min \Bigl(  \frac{(m y)^{1/4}}{f^{3/4}}, \frac{t^{1/3}}{g^{2/3}} \Bigr)
  \end{equation}
By rescaling, it suffices to show that a graph has this value of $\chi_{f} (G)$ and $O(m)$ edges, $O(t)$ triangles, and $O(y)$ local triangle bound. Also, as we have discussed above, we can apply Proposition~\ref{kim-thm2} with real numbers $k, i \ga 1$. We will break this into a number of cases.

\textbf{Case I: $\bm{y < (m \log m)^{1/3}}$.} Then take $G$ to be the triangle-free graph with $m$ edges and $\chi_{f}(G) \ga \frac{m^{1/3}}{\log^{2/3} m}$. This satisfies (\ref{sas1}), as:
$$
\frac{ (m y)^{1/4} }{f^{3/4}} \leq \frac{ (m (m \log m)^{1/3})^{1/4} }{\log( \frac{m}{(m \log m)^{1/3}})^{3/4}} \approx \frac{m^{1/4} m^{1/12} \log^{1/12} m}{\log^{3/4} m} = \frac{m^{1/3}}{\log^{2/3} m}
$$

\textbf{Case II: $\bm{y > t^{2/3}}$.} In this case, we have $g = 1$. We take $G$ to be a clique on $t^{1/3}$ vertices. This graph has $\chi_{f}(G) = t^{1/3} = \frac{t^{1/3}}{g^{2/3}}$. Furthermore, it has $t$ triangles, local triangle bound $t^{2/3} \leq y$, and $t^{2/3} \leq m$ edges.

\textbf{Case III: $\bm{(m \log m)^{1/3} \leq y \leq t^{2/3}}$ and $\bm{t \geq m^{3/4} y^{3/4} / f^{1/4}}$.} Then apply Proposition~\ref{kim-thm2} with
$$
i = \frac{y^{3/4}}{(m f)^{1/4}} \qquad k = m/y
$$
and let $G = H_{k,i}$. The bound $y \leq t^{2/3} \leq m$ shows that $k \geq 1$ and the bound $y \geq (m \log m)^{1/3}$ shows that $i \geq 1$.

Note that $\log k = f$. So by (B5), we have $\chi_{f}(G) \ga \frac{i \sqrt{k}}{\sqrt{\log k}} = \frac{(m y)^{1/4}}{f^{3/4}}
$ as required by (\ref{sas1}).

By (B2), $G$ has local triangle bound $O(i^2 \sqrt{k \log k}) = O(y)$. By (B3), it has $O(k^{3/2} i^2 \sqrt{\log k}) = O(m)$ edges. The overall triangle count is $O(i k y) = O(\frac{m^{3/4} y^{3/4}}{f^{1/4}})$; by our hypothesis on the size of $t$, this is $O(t)$.

\textbf{Case IV: $\bm{(m \log m)^{1/3} \leq y \leq t^{2/3}}$ and $\bm{t < m^{3/4} y^{3/4} / f^{1/4}}$}. We begin by showing that $f \ga g$ in this case. Since $m > f^{1/3} t^{4/3} / y$, we have
$$
f = \log( \frac{m}{y} ) \geq \log( \frac{ f^{1/3} t^{4/3}/y}{y} ) = \log( \log^{1/3}(t^2/y^3) (t^2/y^3)^{2/3} ) \approx \log(t^2/y^3) = g
$$

In this case, we apply Proposition~\ref{kim-thm2} with
$$
i = \frac{y}{(g t)^{1/3}}, k = \frac{g^{1/3} t^{4/3}}{y^2}
$$
and take $G = H_{k,i}$.

We have $k \geq 1$ since $y \leq t^{2/3}$. To show $i \ga 1$, we use the upper bound on $t$ and lower bound on $y$, and the estimate $f \ga g$ to give:
$$
i \geq \frac{y}{g^{1/3} (m^{3/4} y^{3/4} / f^{1/4})^{1/3}} = \frac{f^{1/12} y^{3/4}}{m^{1/4} g^{1/3}} \ga \frac{g^{1/12} (m \log m)^{1/4}}{m^{1/4} g^{1/3}} = \frac{\log^{1/4} m}{g^{1/4}} \ga 1
$$

We also estimate $\log k$ in this case:
$$
\log k = \log( g^{1/3} t^{4/3}/y^2 ) = \log( \log^{1/3}(t^2/y^3)  (t^2/y^3)^{2/3} ) \approx \log( t^2/y^3) = g
$$

Thus, by (B2), the graph $G$ has local triangle bound $O(i^2 \sqrt{k \log k}) \la y$. By (B1) and (B2), it has $O(k i y) = O(t)$ triangles total. Finally, by (B3), it has $O(k^{3/2} i^2 \sqrt{\log k})$ edges, which we estimate as
$$
k^{3/2} i^2 \sqrt{\log k} \la \frac{g^{1/3} t^{4/3}}{y}
$$
Using the bound $g \la f$ and $t \leq m^{3/4} y^{3/4}/f^{1/4}$, we see this is at most $O(m)$.

Finally, by (B5), it has
$$
\chi_f(G) \ga \frac{i \sqrt{k}}{\sqrt{\log k}} \approx \frac{y \sqrt{g^{1/3} t^{4/3}/y^2}}{\sqrt{g}} = \frac{t^{2/3}}{g^{1/3}}
$$
thus satisfying (\ref{sas1}).
\end{proof}
\section{Getting the correct coefficient of $t^{1/3}$}
Suppose we have no information on the local triangle counts; in this case, Theorem~\ref{ttprop2} would give the following bound in terms of the vertex count $n$ and global triangle count $t$ alone:
$$
\chi(G) \leq O( t^{1/3} + \sqrt{\frac{n}{\log n}} )
$$
This bound is clearly tight up to constant factors. In this section, we compute a more precise formula, which gives us the correct coefficient of the term $t^{1/3}$. The correct coefficient of the term $\sqrt{\frac{n}{\log n}}$ is not currently known even for triangle-free graphs.

We begin by showing a more precise bound, albeit with a worse asymptotic dependence on $n$.
\begin{proposition}
\label{tw-prop1}
If $G$ has $n$ vertices and $t$ triangles, then $\chi(G) \leq 2 \sqrt{n} + (6 t)^{1/3}$.
\end{proposition}
\begin{proof}

Let us define $f(n,t) = 2 \sqrt{n} + (6 t)^{1/3}$, and let $d = \lfloor f(n, t) \rfloor$. Suppose that some vertex $v$ of $G$ has degree at most $d-1$. In that case, apply the induction hypothesis to $G - v$, obtaining a coloring using at most $\lfloor f(n-1, t) \rfloor \leq d$ colors. At least one color in the range $\{1, \dots, d \}$ is not used by any a neighbor of $v$, so this coloring can be extended to $v$ as well.

So suppose that $G$ has minimum degree at least $d$. Among all vertices, select one vertex $v$ which participates in the \emph{minimum} number $y$ of triangles. The graph $G[N(v)]$ must have $y$ edges and at least $d$ vertices. By Theorem~\ref{turan-thm}, $G[N(v)]$ contains an independent set $I$ of size at least $s = \frac{d}{2 y/d + 1}$. Assign all vertices in $I$ one new color, and recurse on $G - I$.

As $v$ is chosen to be incident on the minimum number of triangles, every vertex of $I$ is itself incident on at least $y$ triangles; furthermore as $I$ is independent these triangles are all distinct. So $I$ is incident on at least $y s$ triangles, which are all removed in $G - I$. By induction hypothesis $\chi(G - I) \leq f(n - s, t - y s)$ so
\begin{align*}
\chi(G) &\leq 1 + f(n - s, t - y s) = 1 + 2 \sqrt{n - s} + 6^{1/3} (t - y s)^{1/3} \\
&\leq 1 + 2 \sqrt{n} - \frac{s}{\sqrt{n}} + 6^{1/3} (t^{1/3} - \frac{y s}{3 t^{2/3}}) = f(n,t) + 1 - \frac{s}{\sqrt{n}} - \frac{6^{1/3} y s}{3 t^{2/3}}
\end{align*}

We want to show that $\chi(G) \leq f(n,t)$; thus, it suffices to show that
\begin{equation}
\label{yy1}
\frac{s}{\sqrt{n}} + \frac{6^{1/3} y s}{3 t^{2/3}} \geq 1
\end{equation}

Substituting in the value for $s$, simple algebraic manipulations show that this is equivalent to showing:
\begin{equation}
\label{yy2}
d (\frac{d}{\sqrt{n}} - 1 ) + y \Bigl( \frac{6^{1/3} d^2}{3 t^{2/3}} - 2 \Bigr) \geq 0
\end{equation}

In order to show that (\ref{yy2}) holds, note that $d \geq \lfloor 2 \sqrt{n} \rfloor \geq \sqrt{n}$ and so we have $d (\frac{d}{\sqrt{n}} - 1 ) \geq 0$. Also, we have $d \geq 2 \sqrt{n} + (6 t)^{1/3} - 1  \geq (6 t)^{1/3}$, and so $\frac{6^{1/3} d^2}{3 t^{2/3}} - 2 \geq \frac{6^{1/3}( 6 t)^{2/3} }{3 t^{2/3}} - 2 = 0$. This completes the induction.
\end{proof}

\begin{theorem}
\label{t-prop2}
Suppose that $G$ contains $t$ triangles and $n$ vertices. Then
$$
\chi(G) \leq O \Bigl( \sqrt{\frac{n}{\log n}} + \frac{t^{1/3} (\log \log n)^{3/2}}{\log n} \Bigr) + (6 t)^{1/3} = O( \sqrt{\frac{n}{\log n}} ) + (6^{1/3} + o(1)) t^{1/3}
$$
\end{theorem}
\begin{proof}
  If $t \leq (n/\log n)^{3/2}$, then this follows immediately from Theorem~\ref{ttprop2}. So let us suppose that $t > (n/\log n)^{3/2}$.
  
  Let $y = t^{1/3} \log^2 n$. Let $A$ denote the vertices of $G$ which are incident on at least $y$ triangles, and let $a$ be the number of triangles contained in $G[A]$. Similarly let $B = V - A$ and let $b$ be the number of triangles in $G[B]$. Note that $|A| \leq 3 t/y$ and $a + b \leq t$.

  By applying Proposition~\ref{tw-prop1} to $G[A]$ and Theorem~\ref{ttprop2} to $G[B]$, we get
  \begin{align*}
  \chi(G) &\leq \chi(A) + \chi(B) \leq 2 \sqrt{3t/y} + (6 a)^{1/3} + O(\sqrt{\frac{n}{\log n}} + \frac{b^{1/3} \log \log((b^2/y^3)}{\log^{2/3}(b^2/y^3)}) \\
  &= (6 a)^{1/3} + O \Bigl( \sqrt{\frac{n}{\log n}} + \frac{b^{1/3} \log \log(\frac{b^2}{t \log^6 n})}{\log^{2/3}(\frac{b^2}{t \log^6 n})} + \frac{t^{1/3}}{\log n} \Bigr)
\end{align*}
   As $a + b \leq t$, we have $(6 a)^{1/3} \leq (6 (t-b))^{1/3} \leq (6 t)^{1/3} - \frac{2^{1/3} b}{3^{2/3} t^{2/3}}$. Thus, collecting all relevant constant terms, we have shown
\begin{equation}
\label{b-e0}
\chi(G) \leq (6 t)^{1/3} + O \Bigl( \frac{t^{1/3}}{\log n} + \sqrt{ \frac{n}{\log n} }  + \frac{b^{1/3} \log \log ( \frac{b^2}{t \log^6 n}) }{\log^{2/3}(\frac{b^2}{t \log^6 n})} - \frac{C b}{t^{2/3}} \Bigr)
\end{equation}
for some constant $C > 0$.

Our next task is to show
\begin{equation}
\label{b-e1}
 \frac{b^{1/3} \log \log ( \frac{b^2}{t \log^6 n})}{\log^{2/3}( \frac{b^2}{t \log^6 n})} -\frac{C  b}{t^{2/3}} \la \frac{t^{1/3} (\log \log n)^{3/2}}{\log n}
\end{equation}

If $b \leq t/\log^3 n$, then
$$
\frac{b^{1/3} \log \log ( \frac{b^2}{t \log^6 n})}{\log^{2/3}( \frac{b^2}{t \log^6 n})} \la b^{1/3}  \la \frac{t^{1/3}}{\log n}
$$
and we have shown (\ref{b-e1}). Otherwise, if $b \geq t/\log^3 n$, then $\log( \frac{b^2}{t \log^6 n} ) \geq \log( \frac{t}{\log^{1/12} n} ) \ga \log n$, using our assumption that $t> (n/\log n)^{3/2}$. Thus, in this case, it suffices to show:
\begin{equation}
\label{b-e2}
\frac{b^{1/3} \log \log n}{\log^{2/3} n} - C b / t^{2/3} \la \frac{t^{1/3} (\log \log n)^{3/2}}{\log n}
\end{equation}

The LHS of (\ref{b-e2}) is negative unless $b \leq \frac{ (\log \log n)^{3/2} t}{C^{3/2}}$; when $b$ is in this range then $\frac{b^{1/3} \log \log n}{\log^{2/3} n} \la \frac{t^{1/3} (\log \log n)^{3/2}}{\log n}$ as desired.
\end{proof}

\begin{theorem}
\label{t-prop3}
Suppose that $G$ contains $t$ triangles and $m$ edges. Then
$$
\chi(G) \leq O \Bigl( \frac{m^{1/3}}{\log^{2/3} m} + \frac{t^{1/3} (\log \log m)^{3/2}}{\log m} \Bigr) + (6 t)^{1/3} = O \Bigl( \frac{m^{1/3}}{\log^{2/3} m} \Bigr) + (6^{1/3} + o(1)) t^{1/3}
$$
\end{theorem}
\begin{proof}
The proof is nearly identical to Theorem~\ref{t-prop2}, using Theorem~\ref{ttprop3} instead of Theorem~\ref{ttprop2}.
\end{proof}

By considering a clique of $(6 t)^{1/3}$ vertices or a triangle-free graph, we can easily see that Theorems~\ref{t-prop2} and~\ref{t-prop3} are tight up to lower-order terms.

\section{Conjectured tight bounds}
\label{conj-sec}
The following conjecture seems natural:
\begin{conjecture}
\label{conj1}
Suppose $G$ has $n$ vertices, $t$ triangles, and local triangle bound $y$. Then
$$
\chi(G) \la \frac{t^{1/3}}{\log^{2/3} (t^2/y^3) } + \sqrt{ \frac{n}{\log n}}
$$
\end{conjecture}

This conjecture would strengthen Theorem~\ref{ttprop2}, would give Proposition~\ref{prop0} as a special case (as $t \leq n y$), and would match the lower bound Proposition~\ref{lb0}.

As further evidence for Conjecture~\ref{conj1}, we show that it relates to an natural conjecture on degeneracy and fractional chromatic number in triangle-free graphs:
\begin{conjecture}
\label{fconj}
Suppose that $G$ is $d$-degenerate and triangle-free. Then $\chi_{f}(G) \la \frac{d}{\log d}$.
\end{conjecture}

In a draft version of this paper, we had formulated a more general conjecture that fractional chromatic number is bounded by a constant times Hall ratio. One can easily show a $d$-degenerate triangle-free graph has Hall ratio $O(\frac{d}{\log d})$, and so this conjecture would have immediately shown Conjecture~\ref{fconj}. This broader conjecture has subsequently been refuted \cite{refute-paper, refute-paper2}, but we believe that it may still hold in certain restricted graph classes or settings. Even if this does not hold in general, it may still be some heuristic justification for Conjecture~\ref{fconj}. 

As two other pieces of evidence for Conjecture~\ref{fconj}, we note that \cite{aks} showed that a triangle-free, $d$-degenerate graph may have $\chi(G)$ as large as $d$; the graph $G$ which achieves this indeed has $\chi_{f}(G) \la \frac{d}{\log d}$. Finally, we note that \cite{esperet} has found other applications of Conjecture~\ref{fconj}, which would cleanly show certain other bounds regarding induced bipartite graphs. Slightly weaker versions of these bounds have been found by other more laborious methods.

We now show a number of consequences of Conjecture~\ref{fconj}:
\begin{lemma}
  \label{ttprop1cc}
  Suppose that  $G$ is $d$-degenerate and has local triangle bound $y$.  If Conjecture~\ref{fconj} holds, then
  $$
  \chi_f(G) \la \frac{d}{\log(d^2/y)}
  $$
\end{lemma}
\begin{proof}
  Consider some weighting function $w: V \rightarrow \mathbb R_{+}$. We first claim that there is a vertex set $U \subseteq V$ such that $G[U]$ is $d' = O( d/\sqrt{y})$-degenerate, is triangle-free, and has $w(U) \ga w(V)/\sqrt{y}$. We show this by a randomized construction: let us first choose some orientation $J$ of $G$, such that every vertex in $v$ has at most $d$ out-neighbors with respect to $J$.

  Now consider forming a subset $W \subseteq V$, wherein each $v \in V$ goes into $W$ independently with probability $p = \frac{1}{10 \sqrt{y}}$. We then let $U$ denote the set of vertices $v \in Y$ such that $v$ has no triangles in $G[W]$ and at most $d/\sqrt{y}$ of the out-neighbors of $v$ (with respect to $J$) are in $Y$.
    A simple Markov's inequality calculation shows that, conditional on $v \in W$, the vertex $v$ survives to $U$ with probability $\Omega(1)$. Therefore, $\bE[ w(U)] \ga \frac{w(V)}{\sqrt{y}}$. Also, by construction, the graph $G[U]$ is triangle-free. Since every vertex in $U$ has at most $d/\sqrt{y}$ out-neighbors with respect to $J$, $G[U]$ is also $d/\sqrt{y}$-degenerate. In particular, there exists a vertex set $U$ that satisfies the desired properties.

    Now apply Conjecture~\ref{fconj} to $G[U]$, obtaining an independent set $I$ with
    $$
    w(I) \ga \frac{w(U) \log(d/\sqrt{y}) }{ d/\sqrt{y} } = \frac{ \frac{w(V)}{\sqrt{y}} \log(d/\sqrt{y})}{d / \sqrt{y}} \approx \frac{w(V) \log(d^2/y)}{d}
      $$
\end{proof}
  
\begin{lemma}
  \label{ttprop1aa}
  Suppose $G$ has $n$ vertices, $t$ triangles, and local triangle bound $y$. If Conjecture~\ref{fconj} holds, then $G$ has an independent set $I$ satisfying either (i) $|I| \geq \Omega(\sqrt{n \log n})$ or (ii) $I$ touches  $\Omega(t^{2/3} \log^{2/3}(\frac{t^2}{y^3}))$ triangles in $G$.
\end{lemma}
\begin{proof}
  Let us set $f = \log(t^2/y^3)$ and $d = (t f)^{1/3} + \sqrt{n \log n}$. Let us consider a number of cases.
  
\textbf{CASE I: $\bm{G}$ is not $\bm{d}$-degenerate.} Then there is a non-empty vertex set $U \subseteq V$ such that $G[U]$ has minimum degree at least $d$. Let $w \in U$ be the vertex of $U$ which is incident upon the fewest triangles of $G$; suppose that $w$ is incident upon $k$ triangles. By Theorem~\ref{turan-thm}, there is an independent set $I \subseteq N(w) \cap U$ with $|I| \geq \frac{d}{1 + \frac{2 k}{d}}$.

If $k \leq d$, then $|I| \ga d \geq \sqrt{n \log n}$. Otherwise, if $k \geq d$, then $|I| \ga d^2/k$. Since $w$ is incident on the smallest number of triangles of $U$, every vertex in $I$ must be incident upon at least $k$ triangles. So $I$ is incident upon at least $|I| k \ga d^2 \geq (t f)^{2/3}$ triangles, as desired.

\textbf{CASE II: $\bm{G}$ is $\bm{d}$-degenerate and $\bm{\sqrt{n \log n} < (t f)^{1/3}}$.} Define the weighting function $w: V \rightarrow \mathbb R$ by setting $w(v)$ to be the number of triangles incident upon $v$. By Lemma~\ref{ttprop1cc}, there is an independent set $I$ with $\sum_{v \in I} w(v) \ga \frac{\sum_{v \in V} w(v)}{d/\log(d^2/y)} \ga \frac{t \log(d^2/y)}{d}$. We compute the log term here as
$$
\log(d^2/y) \geq \log( (t f)^{2/3} / y ) = \log( (t^2/y^3)^{1/3} \log^{2/3} (t^2/y^3) ) \approx f
$$
Using the bound $(t f)^{1/3} > \sqrt{n \log n}$, we thus estimate
$$
\frac{t \log(d^2/y)}{d} \ga \frac{t f}{(t f)^{1/3}} \ga (t f)^{2/3}
$$

\textbf{CASE III: $\bm{G}$ is $\bm{d}$-degenerate and $\bm{\sqrt{n \log n} \geq (t f)^{1/3}}$.} Let $A$ denote the vertices in $G$ incident on at most $10 t/n$ triangles; we must have $|A| \ga n$. Since $G$ is $d$-degenerate, the graph $G[A]$ is also $d$-degenerate and has local triangle bound $y' = 10 t/n$. Define the weighting function $w: V \rightarrow \mathbb R$ by setting $w(v) = 1$. By Lemma~\ref{ttprop1cc}, there is an independent set $I \subseteq A$ with $|I| = \sum_{v \in I} w(v) \ga \frac{\sum_{v \in A} w(v)}{d/\log(d^2/y')} \ga \frac{n \log(n d^2/t)}{d}$. We compute the log term here as
$$
\log(n d^2/t) \geq \log(n \times n \log n / (n^{3/2}/f)) = \log( f \sqrt{n} \log n) \ga \log n
$$

Also, in this case, we have $d \leq \sqrt{n \log n}$, and so overall $|I| \ga \sqrt{n \log n}$.
\end{proof}

\begin{proposition}
  \label{ttprop1a}
Conjecture~\ref{fconj} implies Conjecture~\ref{conj1}.
\end{proposition}
\begin{proof}
We will show that
\begin{equation}
\label{trt2}
\chi(G) \leq C \Bigl( \sqrt{\frac{n}{\log n}} + \frac{t^{1/3}}{\log^{2/3}(t^2/y^3)} \Bigr)
\end{equation}
by induction on $n$, where $C$ is some sufficiently large universal constant. Let us set $f = \log(t^2/y^3)$. We may assume that $n, t$, and $f$ are larger than any needed constant; otherwise, this follows from Theorem~\ref{ttprop2} if the constant $C$ is sufficiently large.

Let $I$ be the independent set guaranteed by Lemma~\ref{ttprop1aa}. We assign all the vertices in $I$ one new color, remove $I$ from the graph, and apply our induction hypothesis to color $G - I$.

First, suppose that $I$ is incident upon at least $r = c (t f)^{2/3}$ triangles for some constant $c \leq 1$. These triangles are all removed in $G - I$, and so by induction hypothesis we have
\begin{equation}
  \label{tpp1}
\chi(G) \leq 1 + C ( \sqrt{\frac{n}{\log n}} + \frac{(t - r)^{1/3}}{\log^{2/3}((t - r)^2/y^3)} )
\end{equation}

 Now consider the function $F(x) = x^{1/3}/\log^{2/3} (x^2/y^3)$. Simple calculus shows that this is an increasing concave-down function as long as $x \geq a y^{3/2}$ for some constant $a > 0$.  
 Using the fact that $r = c (t f)^{2/3} \leq (t \log t)^{2/3}$ and that $t,r,f$ are all larger than any needed constants, we observe that $t - r \geq t/2 \geq \frac{ e^{f/2} }{2} y^{3/2} \geq a y^{3/2}$. Therefore, we may bound the expression in the RHS of (\ref{tpp1}) by its tangent lines, i.e.
$$
\chi(G) \leq 1 + C ( \sqrt{\frac{n}{\log n}} + F(t - r) ) \leq 1 + C ( \sqrt{\frac{n}{\log n}} + F(t) - r F'(t)) = 1 + C( \sqrt{\frac{n}{\log n}} + \frac{t^{1/3}}{f^{2/3}} - c (t f)^{2/3} \frac{ f - 4 }{t^{2/3} f^2} )
$$

To show that $\chi(G) \leq C ( \sqrt{ \frac{n}{\log n}} + \frac{t^{1/3}}{f^{2/3}})$ it thus suffices to show that $1 - C c \frac{(t f)^{2/3} (f - 4)}{3 t^{2/3} f^{5/3}}  \leq 0$. By taking $C$ to be sufficiently large, this follows immediately from our assumption that $f$ is larger than any needed constant. Thus the induction   holds in this case.

Next, if $|I| \geq c \sqrt{n \log n}$ for constant $c$ this gives
$$
\chi(G) \leq 1 + C ( \sqrt{ \frac{n - c \sqrt{n \log n}}{\log(n - c \sqrt{n \log n})}} + \frac{t^{1/3}}{\log^{2/3}(t^2/y^3)} )
$$

By a similar tangent line argument, we see that
$$
\chi(G) \leq 1 + C ( \sqrt{\frac{n}{\log n}} + \frac{t^{1/3}}{\log^{2/3}(t^2/y^3)} ) - \frac{C c (\log n - 1)}{2 \log n}
$$
This is at most $C (\sqrt{\frac{n}{\log n}} + \frac{t^{1/3}}{\log^{2/3}(t^2/y^3)} )$ for $C$ a sufficiently large constant, and the induction again proceeds in this case.
\end{proof}

From Proposition~\ref{ttprop1a}, we can tighten many of the bounds in our paper. We omit the proofs since they are essentially identical to proofs we have already encountered.
\begin{proposition}
Suppose $G$ has $n$ vertices, $m$ edges, $t$ triangles, and local triangle bound $y$. If Conjecture~\ref{fconj} holds then we have the following bounds:
\begin{enumerate}
\item $\chi(G) \la \frac{t^{1/3}}{\log^{2/3}(t^2/y^3)}  + \frac{m^{1/3}}{\log^{2/3} m}$
\item $\chi(G) \leq O \Bigl( \sqrt{\frac{n}{\log n}} + \frac{t^{1/3}}{\log n} \Bigr) + (6 t)^{1/3} = O( \sqrt{\frac{n}{\log n}} ) + (6^{1/3} + o(1)) t^{1/3}$
\item  $\chi(G) \leq O \Bigl( \frac{m^{1/3}}{\log^{2/3} m} + \frac{t^{1/3} }{\log m} \Bigr) + (6 t)^{1/3} = O \Bigl( \frac{m^{1/3}}{\log^{2/3} m} \Bigr) + (6^{1/3} + o(1)) t^{1/3}$
\end{enumerate}

\end{proposition}

\section{Acknowledgments}
Thanks to Vance Faber, Paul Burkhardt, Louis Ibarra, Aravind Srinivasan, and anonymous journal reviewers for helpful suggestions,  discussions, and proofreading. Thanks to Seth Pettie for some clarifications about Theorem~\ref{triangle-free-d}. Thanks to Tran Manh Tuan for clarification about the paper \cite{poljak-tuza}. Thanks to Hsin-Hao Su for some discussions about fractional chromatic number.


\begin{thebibliography}{1}

\bibitem{ajtai}
Ajtai, M., Koml\'{o}s, J., Szemer\'{e}di, E.: A note on Ramsey numbers. Journal of Combinatorial Theory Series A 29, pp. 354-360 (1980)

\bibitem{aks}
Alon, N., Krivelevich, M., Sudakov, B.: Coloring graphs with sparse neighborhoods. Journal of Combinatorial Theory, Series B 77.1, pp. 73-82 (1999)

\bibitem{refute-paper} 
Blumenthal, A., Lidicky, B., Martin, R., Norin, S., Pfender, F., Volec, J.: Counterexamples to a conjecture of Harris on Hall ratio. arxiv:1811.1116 (2018)

\bibitem{bohman}
Bohman, T., Mubayi, D.: Independence number of graphs with a prescribed number of cliques. arxiv:1801.01091 (2018)

\bibitem{refute-paper2}
  Dvo\u{r}\'{a}k, Z., Ossona de Mendez, P., Wu, H.: $1$-subdivisions, fractional chromatic number and Hall ratio. arxiv: 1812:07327 (2018)
  
\bibitem{esperet}
Esperet, L., Kang, R., Thomass\'{e}, S.: Separation choosability and dense bipartite induced subgraphs. arxiv: 1802.03727 (2018)


\bibitem{gimbel}
Gimbel, J., Thomassen, C.: Coloring triangle-free graphs with fixed size. Discrete Mathematics 219-1, pp. 275-277 (2000)


\bibitem{kwan}
Kwan, M., Letzter, S., Sudakov, B., Tran, T.: Dense induced bipartite subgraphs in triangle-free graphs. arxiv: 1810.12144 (2018)

\bibitem{kim}
Kim, J. H.: The Ramsey number $R(3,t)$ has order of magnitude $t^2/\log t$. Random Structures \& Algorithms 7-3, pp. 173-207 (1995).

\bibitem{molloy}
Molloy, M., Reed, B.: Graph colouring and the probabilistic method. Algorithms and Combinatorics, Springer (2001)

\bibitem{molloy2} 
Molloy, M.: The list chromatic number of graphs with small clique number. arxiv:1701.09133 (2017)

\bibitem{nilli}
Nilli, A.: Triangle-free graphs with large chromatic number. Discrete Mathematics 211-1, pp. 261-262 (2000)

\bibitem{pettie}
Pettie, S., Su, H.: Distributed coloring algorithms for triangle-free graphs. Information and Computation 243, pp, 263-280 (2015)

\bibitem{poljak-tuza}
Poljak, S., Tuza, Z.: Bipartite subgraphs of triangle-free graphs. SIAM Journal on Discrete Math 7-2, pp. 307-313 (1994)

\bibitem{vu}
Vu, V.: A general upper bound on the list chromatic number of locally sparse graphs. Combinatorics, Probability, and Computing 11, pp. 103-111 (2002)

\end{thebibliography}
\end{document}